\documentclass[12pt,a4paper]{article}
\usepackage[utf8]{inputenc}
\usepackage{amsmath}
\usepackage{amsfonts}
\usepackage{amssymb}
\usepackage{graphicx}
\usepackage{mathabx}
\usepackage{bbm}
 \usepackage{pdfsync}
\usepackage[english]{babel}
\usepackage[colorlinks]{hyperref}

\usepackage{makeidx}         
\usepackage{multicol}        
\usepackage[bottom]{footmisc}
\usepackage{mathabx}

    \usepackage{pdfsync}

\DeclareMathOperator{\Pd}{\mathbb{P}}
\DeclareMathOperator{\Ed}{\mathbb{E}}
\DeclareMathOperator{\un}{1\hspace{-.29em}I}

\newtheorem{lemma}{Lemma}

\newtheorem{proposition}{Proposition}
\newtheorem{remark}{Remark}[section]

\newenvironment{proof}[1][Proof]{\textbf{#1.} }{\ \rule{0.5em}{0.5em}}
\author{Yves Le Jan}
\title{Loop clusters on complete graphs } 
\begin{document}
\maketitle

\footnotetext{ Key words and phrases: Markov loops, complete graphs}
\footnotetext{  AMS 2010 subject classification:  60C05, 60J27, 60G60.}
\begin{abstract}
We investigate  random partitions of complete graphs defined by Poissonian emsembles of Markov loops. \\

\end{abstract}

\section{Loop clusters}

Let us first review briefly a few general results about loop clusters defined on a general graph $\mathcal{G}=(X,E)$, equiped with a set of conductances  $C_{x,y}$ and a nonnegative killing function $\kappa$. This review is based on references \cite{stfl}, \cite{ljlm} and \cite{book}. Our setup allows to define a transition matrix $P_{x,y}=\frac{C_{x,y}}{\lambda_x}$, with $\lambda_x=\sum_y C_{x,y}+\kappa_x$ and a measure $\nu$ on discrete oriented loops (see section 2.1 in \cite{book}): the measure of a loop is given by the product of the transition matrix values of his edges divided by its multiplicity. Then, one defines for each positive $\alpha$, a Poisson process $\mathcal{DL}_{\alpha}$ of intensity $\alpha\nu$. The size of a loop configuration (i.e. a multiset of discrete loops) is defined to be the sum the lengths of its loops.\\ Occupied edges define a partition $\mathcal{C}_{\alpha}$ of the graph into connected components, possibly including some isolated vertices.\\
Alternatively, the partition $\mathcal{C}_{\alpha}$ can be defined using a random set $\mathcal{PL}_{\alpha}$ of primitive discrete loops (a discrete loop $\xi$ is said to be \emph{primitive} if it does not have a non-trivial period). A primitive loop $\eta$ belongs to $\mathcal{PL}_{\alpha}$ with probability $1-(1-\nu(\eta_i))^{\alpha}$ ($=\nu(\eta_i)$ if $\alpha=1$) and this happens independently for all primitive loops. These two constructions are in fact simply related. Any discrete loop $\xi$ is a multiple of a primitive discrete loop which
we can denote by $\pi\xi$. Then, $\mathcal{PL}_{\alpha}$ can be identified with $ \pi(\mathcal{DL}_{\alpha})$  (See section 5.1 in \cite{book}).

If we assume that the graph $\mathcal{G}$ is finite, and that $\kappa$ does not vanish everywhere, the Green matrix $G$ is defined to be the inverse of $\lambda_x\delta_{x,y}-C_{x,y}$. Moreover, given any subset $D$ of $X$, the Green matrix $G^{D}$ is the inverse of the restriction of $\lambda_x\delta_{x,y}-C_{x,y}$ to $D\times D$ ($X$ is replaced by $D$ and $\kappa$ by $\kappa+\sum_{y\in X-D}C_{x,y}).$  

The probability that $\mathcal{C}_{\alpha}$ is thinner ($ \succeq$) than a given partition $\pi$  has a simple expression in terms of the Green matrices (See \cite{ljlm} or section 5.3 in \cite{book}): 
\begin{proposition}\label{semigrouppartition}

Given a partition $\pi=(B_i)_{i\in I}$ of $X$,
 \begin{equation}\label{semigroupdetG}
\Pd(\mathcal{C}_{\alpha}\succeq \pi)=\left(\frac{\prod_{i\in I}\det(G^{B_i})}{\det(G)}\right)^{\alpha}. 
\end{equation}
\end{proposition}

Using an inclusion exclusion argument, an explicit formula for $\Pd(\mathcal{C}_{\alpha}=\pi)$ is derived from this proposition in \cite{ljlm}.  Let us first introduce some notations. For  a partition $\pi$,  let $|\pi|$ denote the number of non-empty blocks of $\pi$. For a subset $A$ of $X$,  let $\pi_{|A}$ denote the restriction of $\pi$ to $A$: $\pi_{|A}$ is a partition of $A$, 
the blocks of which are the intersection of the blocks of $\pi$ with $A$. 
\begin{proposition}

Let $\pi$ be a partition of $X$ with $k$ non-empty blocks denoted by $B_1,\ldots,B_k$. Then
\begin{equation}
\label{sgr_alpha}
\Pd(\mathcal{C}_{\alpha}=\pi)=\sum_{\tilde{\pi}\succeq \pi}(-1)^{|\tilde{\pi}|-k}\prod_{i=1}^{k}(|\tilde{\pi}_{|B_i}|-1)!\Pd(\mathcal{C}_{\alpha}\succeq \tilde{\pi}).
\end{equation}
\end{proposition}
\section{Clusters in the complete graph}
Let us consider the case of the complete graph $K_n$ endowed with unit conductances and a constant killing factor $\kappa$. The expression of the Green matrix $G$ is:
\[
\frac{1}{\kappa (n+\kappa)}(\kappa I+J)
\]
where $J$ denotes the $(n,n)$ matrix with all entries equal to $1$. Similarly, if $|D|=d$, the expression of the Green matrix $G^D$ is:
\[
\frac{1}{(n-d+\kappa) (n+\kappa)}((n-d+\kappa) I_D+J_D)
\]
where $I_D$ and $J_D$ denote the restrictions of the $(n,n)$ matrices $I$ and $J$ to $D\times D$.
Moreover, it is easy to check that for any $(m,m)$  square matrix $M$ with diagonal entries equal to $a+b$ and off diagonal entries equal to $b$, $\det(M)=a^{m-1}(a+mb)$. 
Hence we get the following:
 \begin{lemma}\label{lemdet}
$\det(G)=\frac{1}{\kappa (n+\kappa)^{n-1}}$ and $\det(G^D)$= $\frac{1}{(n-d+\kappa)(n+\kappa)^{d-1}}$. 
 \end{lemma}
By formula (\ref{semigroupdetG}), if  $\pi$ is a partition of the set of vertices $X$ with $l$ blocks $B_1,\ldots,B_l$ then,  $\frac {\prod_i \det(G^{B_i})}{\det(G)}=(\frac{\kappa}{\kappa+n})\prod_{1\leq i\leq l}(1-\frac{|B_i|}{n+\kappa})^{-1}$ and
$$\Pd(\mathcal{C}_{\alpha}\succeq \pi)=(\frac{\kappa}{\kappa+n})^{\alpha}\prod_{i=1}^l(1-\frac{|B_i|}{n+\kappa})^{-\alpha}. $$
%
Let us note that 
 $(1-\frac{j}{n+\kappa})^{-\alpha}$ is the $j$-th moment $m_{j}^{(\alpha)}$ of the random variable 
$Y^{(\alpha)}=\exp(\frac{Z^{(\alpha)}}{n+\kappa})$ where $Z^{(\alpha)}$ denotes a Gamma$(\alpha,1)$-distributed
random variable. In particular, for $\alpha=1$,  $Y^{(1)}$ has density $(n+\kappa)y^{-(n+\kappa+1)}\un_{[1,\infty)}$. Thus, $$\displaystyle{\Pd(\mathcal{C}_{\alpha}\succeq\pi%
)=\frac{\prod_{i=1}^{l}m^{(\alpha)}_{\left|  B_{i}\right|  }}{m_n^{(\alpha)}}}
.$$  
Equivalently, if $D_1, ...,D_k$ are disjoint sets of vertices, with $|D_i|=d_i$, $$\Pd(D_1, ..., D_k \text{ are isolated})=
\Pd(\mathcal{C}_{\alpha}\succeq \{D_1, ...,D_k,X-\bigcup_1^kD_i\}%
)=\frac{\prod_{i=1}^{k} m^{(\alpha)}_{d_{i} }m_{n-\sum_1^k d_i}^{(\alpha)}}{m_n^{(\alpha)}}
.$$  

In particular, given any $k$-tuple of vertices $x_1, ..., x_k$,
$$\Pd(x_1, ..., x_k \text{ are isolated})= \frac{\kappa^{\alpha}}{(k+\kappa)^{\alpha}}(1-1/(n+\kappa))^{-k\alpha}.$$
Denoting by $I_1$ the set of isolated vertices, we can therefore compute the factorial moments of $|I_1|$ by considering among all $k$-samples without replacement, those whose elements are isolated:
\begin{equation} \label{anss}
\Ed(|I_1|(|I_1|-1)...(|I_1|-k+1))=n(n-1)...(n-k+1)\frac{\kappa^{\alpha}}{(k+\kappa)^{\alpha}}(1-1/(n+\kappa))^{-k\alpha}.
\end{equation}

Consider now isolated clusters of finite size $d$.\\

Let $c_n^{(\alpha)}$ denote the $n$-th cumulant of $Y^{(\alpha)}$.  Formula (\ref{sgr_alpha}) and the expression of cumulants  in terms of moments (see for example Section 3.3 in \cite{C}, or formula (1.30) in \cite{P}) yield that:
\begin{equation} \label{anis}
\Pd(\mathcal{C}_{\alpha}=\{X\})=\frac{c_{n}^{(\alpha)}}{m_{n}^{(\alpha)}}.
\end{equation}
Denoting by $I_d$ the set of $\mathcal{C}_{\alpha}$-clusters of size $d$, we can now compute the factorial moments of $|I_d|$, we have:
\begin{proposition}\label{fksh}
$$
\Ed(|I_d|...(|I_d|-k+1))= {n \choose d}{n-d \choose d}...{n-kd+d \choose d}[c_d^{(\alpha}]^k \left(\frac {\kappa}{(kd+\kappa)}\right)^{\alpha}.$$
\end{proposition}
\begin{proof}
Given $D$ any set of vertices, since the sets of loops contained in $D$ and the sets of loops intersecting $X-D$ are independent, the probability for $D$ to be a connected cluster of $\mathcal{C}_{\alpha}$ given that $D$ is isolated is the probability that the cluster decomposition of $D$ defined by the loops of $\mathcal{DL}_{\alpha}$ carried by $D$ is trivial. Replacing $X$ by $D$ and $\kappa$ by $n-d+\kappa$ in formula (\ref{anis}) entails that it equals $\frac{c_{d}^{(\alpha)}}{m_{d}^{(\alpha)}}$. 
More generally, given $D_1, ...,D_k$ disjoint sets of vertices, since the sets of loops contained in each $D_i$ and the set of loops intersecting $X-\bigcup_1^kD_i$ are all independent, the probability for $D_1, ...,D_k$ to be connected given that they are isolated is $\prod_1^k\frac{c_{d_i}^{(\alpha)}}{m_{d_i}^{(\alpha)}}$.\\
Therefore,
$$\Pd(D_1, ..., D_k \text{ are clusters of }\mathcal{C}_{\alpha})
=\frac{\prod_{i=1}^{k} c^{(\alpha)}_{d_{i} }m_{n-\sum_1^k d_i}^{(\alpha)}}{m_n^{(\alpha)}}=\prod_{i=1}^{k} c^{(\alpha)}_{d_{i}}\frac {\kappa^{\alpha}}{(\sum_1^k d_i+\kappa)^{\alpha}}
.$$  It is now easy to conclude as in the case of isolated vertices.
\end{proof}

\section{Asymptotics for small clusters}

Letting $n$ increase to infinity, we get from formula (\ref{anss}) and from the expression of moments in terms of factorial moments using Stirling numbers that 
$$\Ed(|I_1|^k)\sim n^k\left(\frac{\kappa}{k+\kappa}\right)^{\alpha}.$$
Let us note that 
 $\left(\frac{\kappa}{k+\kappa}\right)^{\alpha}$ is the $k$-th moment 
  of the random variable 
$R_{\alpha}=\exp(\frac{-Z_{\alpha}}{\kappa})$ where $Z_{\alpha}$ denotes a Gamma$(\alpha,1)$-distributed
random variable. In particular, for $\alpha=1$,  $R_{1}$ has density $\kappa x^{\kappa-1}\un_{[0,1]}$.\\
Consequently, we have:
\begin{proposition}
The proportion $|I_1|/n$ of isolated vertices converges in distribution towards $R_{\alpha}$.
\end{proposition}

Similarly, we wish to get from proposition \ref{fksh} asymptotics for the distribution of $|I_d|$. 
\begin{lemma} \label{rnafu}
As $n$ increases to infinity, $c_d^{(\alpha)}$ is equivalent to $\alpha(d-1)!n^{-d}$.
\end{lemma}\label{mzk}
\begin{proof}
This result is suggested by explicit calculations of $c_d^{(\alpha)}$ for small $d$, but will be proved by considering the probability $p_C$ for an isolated set of size $d$ to be connected (as it is clear that $\lim_{n\rightarrow\infty}m_{d}^{(\alpha)}=1$). \\
Firstly, we note that $p_C$ is larger than the probability $p_-$ that exactly one of the  $(d-1)!$ $d$-gons covering $D$ belongs to $\mathcal{DL}_{\alpha}$. As $P_{x,y}=\frac{1}{n+\kappa-1}$, the $\nu$-  measure of a $d$-gon equals $u_d=(n+\kappa-1)^{-d}$ and $p_ - =(d-1)!\alpha u_de^{-d \alpha u_d}$ It is equivalent to $\alpha(d-1)!n^{-d}$ as $n\rightarrow \infty$.\\
To get an upper bound, let us first assume that $\alpha \geq 1$. $p_C$ is smaller than $p_-$ + the probability $p_>$ that the total size $M$ of the loops of $\mathcal{DL}_{\alpha}$ included in $D$ is larger than $(2d+1)\alpha$ +
 the sum, on all loop configurations included in $D$ with size at least $d+1$ and at most $(2d+1)\alpha$, of their probabilities to be included in $\mathcal{DL}_{\alpha}$ .  As the number of configurations with size at least $d+1$  and at most $(2d+1)\alpha$ is finite (and independent of $n$), this sum $\sigma$ can be bounded by $C(n+\kappa-1)^{-d-1}$, $C$ being some constant. Moreover, the probability $p_>$ can be bounded above by $((n+\kappa-1)/d)^{-(2d+1)\alpha}\Ed(((n+\kappa-1)/d)^M)$.
Applying formula 4.18 in \cite{book} or 6-4 in \cite{stfl} to $G^D$ and to the Green function $\tilde G^D$ defined on $D$ by unit conductances and killing function equal to $\kappa$, we get that $\Ed(((n+\kappa-1)/(d+\kappa-1))^{M})= \left(\frac{\det(\tilde G^D)}{\det(G^D)}\right)^{\alpha}\leq Cn^{\alpha d}$ for some constant $C$. Therefore, $p_>\leq C (n/d)^{-(d+1)\alpha}$.\\

If $\alpha \leq 1$, we note that the probability $p_C$ is smaller than $p_-$ + the probability $p_>$ that the total size $M$ of the loops of $\mathcal{DL}_{\alpha}$ included in $D$ is larger than $(2d+1)$ +
 the sum, on all loop configurations included in $D$ with size at least $d+1$ and at most $(2d+1)$, of their probabilities to be included in $\mathcal{DL}_{\alpha}$.  As the number of configurations with size at least $d+1$  and at most $(2d+1)$ is finite (and independent of $n$), this sum $\sigma$ can be bounded by $C(n+\kappa-1)^{-d-1}$, C being some constant. Note finally that the probability $p_>$ can be bounded above by the one obtained in the case $\alpha=1$.
 \end{proof}

Consequently, from proposition \ref{fksh}, as $n\uparrow \infty$,
$$\Ed(|I_d|(|I_d|-1)...(|I_d-k+1))\rightarrow \alpha^kd^{-k}\left(\frac{\kappa}{kd+\kappa}\right)^{\alpha}.$$
Let us note as before that 
 $\left(\frac{\kappa}{kd+\kappa}\right)^{\alpha}$ is the $k$-th moment 
  of the random variable 
$H_{\alpha}=\exp(\frac{-Z_{\alpha}}{\kappa/d})$ where $Z_{\alpha}$ denotes a Gamma$(\alpha,1)$-distributed
random variable. In particular, for $\alpha=1$,  $H_{1}$ has density $\un_{[0,1]}(x)(\kappa/d)x^{\kappa/d-1}$.\\
Note also that the proof of lemma \ref{mzk} shows that the probability for an isolated set of size $d$ to be connected is asymptotically equivalent to the probability that this set contains no loop except exactly one $d$-gon covering it. Indeed, in both cases $\alpha \leq 1$ and $\alpha \geq 1$, this last probability is obviously smaller than $p_-$ and larger than $p_--\sigma-p_>.$

Consequently:
 \begin{proposition}
The number $|I_d|$ of  $\mathcal{C}_{\alpha}$-clusters of size $d$ is asymptotically equal to the number of isolated $d$-gons. It converges in distribution towards 
a mixture of Poisson distributions:
$$\Pd(|I_d|=k)=\mathbb E( \frac{H_{\alpha}^k}{k!}e^{-H_{\alpha}}).$$
For $\alpha=1$,
$$\Pd(|I_d|=k)-\int_0^1 \frac{x^k}{k!}e^{-x}\frac{\kappa}{d} x ^{\frac{\kappa}{d}-1}dx.$$

\end{proposition}
\begin{remark} For the Erdös-Renyi graph in which edges are open independently with probability $c/n$, we have a very different situation. Using independence, it is easy to check that for any $d\geq 1$ , the $k$-th factorial moment of the number of isolated trees of size $d$ equals $$ {n \choose d}{n-d \choose d}...{n-kd+d \choose d}(c/n)^{k(d-1)}(1-c/n)^{kd(n-kd)+k(k-1)d^2/2+d(d-1)/2-d+1}$$. For each isolated cluster of size $d$, there are $d^{d-2}$ trees spanning it, and we can easily see that, as $n\uparrow \infty$, they provide the leading term in its probability of connectedness. From that, we deduce that the $k$-th moment of the number of isolated clusters of size $d$ is equivalent to $
[d^{d-2}n^d/d!]^k[c/n]^{k(d-1)}e^{-ckd}=[nd^{d-2}c^{(d-1)}e^{-cd}/d!]^k$ so that this number is equivalent to $ \frac{d^{d-2}c^{(d-1)}e^{-cd}}{d!}n$, as shown in \cite{erd}. A much more complete result, including a CLT for fluctuations, was given in \cite{Pitt}.   
  \end{remark}

\begin{remark}   
For $\kappa=1$, the expansion of $\frac{c_{n}^{(1)}}{m_{n}^{(1)}}$ in powers of $z=\frac{1}{n}$ starts with $(d-1)!z^d(1+ d(d-2)z+...$. The second term can be associated with probabilities loop configurations covering $d$ vertices with $d+1$ edges. Moreover, the expansions in powers of $x=\frac{1}{n+1}$ appear to be given by  the triangle A087903 in OEIS. This has still to be understood.
   \end{remark}
\section{Large clusters}
There exists also large clusters. We have the following 
\begin{proposition}

For any positive $\epsilon<1$, the probability that there exists clusters of size larger than $
n^{1-\epsilon}$ converges to $1$ as $n\uparrow\infty$.
\end{proposition}

\begin{proof}

The total mass $|\nu_n|$ of the loop measure $\nu_n$ on $K_n$ is $n\sum_2^{\infty}\frac{1}{k}\frac{(n-1)^{k-1}}{(n-1+\kappa)^k}=\frac{n}{n-1}(-\log(1-\frac{n-1}{n-1+\kappa})-\frac{n-1}{n-1+\kappa})\sim \log(n/\kappa)$ as $n\uparrow \infty$. 
In the same way, it is easily shown in  \cite{Chan}, theorem 1-2, that under the normalized loop measure $\bar\nu_n(dl)$, the distribution of $u(l)=\log(|l|)/\log(n)$ converges towards the uniform distribution on $[0,1]$. 

For any $0<\epsilon <1$, let $H_{\epsilon}$ be the set of loops of length larger than $n^{1-\epsilon}$. $|H_{\epsilon}|$ follows a Poisson distribution of parameter $\alpha\nu_n(H_{\epsilon})>\alpha \log(n)\epsilon/2$ for $n$ large enough. Consequently, with probability larger than $1-n^{-\alpha \epsilon/2}$, there exists at least one loop of length larger than $n^{1-\epsilon}$.

Moreover, demoting by $\{l\}$ the set of vertices traversed by the loop $l$, we have the following identity: $$\bar\nu_(|\{l\}|||l|=x)=  n(1-(1-1/n)^x)\leq x.$$ By Markov inequality: $$\bar\nu_(x-|\{l\}|>m\,|\, |l|=x)\leq (x-n(1-(1-1/n)^x))/m.$$
In particular, taking $x=n^{1-\epsilon}$ and $m=\epsilon n^{1-\epsilon}$, we get: $$\bar\nu_(|\{l\}|<(1-\epsilon)n^{1-\epsilon}\,|\,|l|= [n^{1-\epsilon}])\leq n^{\-\epsilon}/\epsilon.$$ Therefore: $$\bar\nu_(|\{l\}|<(1-\epsilon)n^{1-\epsilon}\,|\,|l|\geq [n^{1-\epsilon}])\leq n^{\-\epsilon}/\epsilon$$ and with probability larger than $(1-n^{-\alpha \epsilon/2})(1-n^{\-\epsilon}/\epsilon)$, there exists at least  one set of vertices traversed by a single loop of size larger than $(1-\epsilon)n^{1-\epsilon}$.\\
As the result holds for sets of vertices traversed by a single loop, it holds obviously for loop clusters.

\end{proof} 

As far as we know, the existence of clusters of size $cn$, for $0<c<1$ is an open question.

    D\'epartement de Math\'ematique. Universit\'e Paris-Saclay.  Orsay\\

\bigskip

\end{document}